\documentclass[12pt,letterpaper]{amsart}

\usepackage{amsmath,amssymb,amsthm}
\usepackage{fullpage}
\usepackage{verbatim}
\usepackage[noadjust]{cite} 
\usepackage{hyperref}
\usepackage[capitalize, nameinlink]{cleveref}
\usepackage{enumitem}
\usepackage[usenames,dvipsnames]{xcolor} 
\usepackage{tikz}
\usetikzlibrary{math}
\usepackage[font=small,labelfont=bf, margin=.5in]{caption}
\usepackage{babel}

\colorlet{prettygreen}{ForestGreen!60!LimeGreen}

\makeatletter
\def\@defaultbiblabelstyle#1{[#1]}
\makeatother

\usetikzlibrary{calc, arrows.meta}
\tikzset{vtx/.style={circle, fill, inner sep=1.5pt}}
\tikzset{openvtx/.style={circle, draw, inner sep=1.5pt}}

\usepackage{tikz-cd}

\newtheorem{theorem}{Theorem}[section]
\newtheorem{lemma}[theorem]{Lemma}
\newtheorem{proposition}[theorem]{Proposition}
\newtheorem{corollary}[theorem]{Corollary}

\newtheorem*{claim*}{Claim}

\theoremstyle{definition}
\newtheorem{definition}[theorem]{Definition}

\newtheorem{example}[theorem]{Example}

\theoremstyle{remark}
\newtheorem{remark}[theorem]{Remark}

\crefname{claim}{Claim}{Claims}

\newlist{homtenum}{enumerate}{1}
\setlist[homtenum,1]{leftmargin=36pt}

\DeclareMathOperator{\id}{id}

\title{On Matsushita $\pi_1^2$ discrete fundamental groups}
\author{Mike Krebs}
\address{Department of Mathematics, California State University --- Los
Angeles}
\email{mkrebs@calstatela.edu}
\author{Alan Pan}
\author{Anand Prakash}
\address{Department of Mathematics, California State University --- Los
Angeles}
\email{aprakas3@calstatela.edu}


\makeatletter
\let\mytitle\@title
\let\myauthor\@author
\makeatother

\begin{document}

\begin{abstract}The Matsushita fundamental groups of a graph $X$, denoted $\pi_1^r(X)$, are certain discrete versions of the fundamental group for topological spaces.  For $r=2$, these groups have a nice combinatorial description, due to Sankar.  In this paper we prove two results about $\pi_1^2$.  First, we prove a Seifert-van Kampen-type theorem.  Similar results have previously been obtained by Barcelo, et al. (and strengthened by Kapulkin and Mavinkurve) for a different notion of discrete fundamental group.  Second, we prove that an arbitrary group $G$ can be realized as $\pi_1^2(X)$ for some graph $X$. Our construction works equally well for the aforementioned alternate discrete fundamental group $A_1(X)$, and our second result thus also provides an entirely different method of proof for a theorem of Kapulkin and Mavinkurve.
\end{abstract}


\maketitle

\section{Introduction}

The fundamental group is a well-known and much-studied invariant of topological spaces.  There are various techniques to compute $\pi_1(X)$ for a given topological space $X$.  For example, the Seifert-van Kampen theorem allows one to find, under appropriate conditions, the fundamental group of a space in terms of fundamental groups of subspaces that cover it.  Another method involves the theory of covering spaces.  These two are perhaps the most foundational techniques for calculating fundamental groups.  With regard to covering spaces, we wish to highlight a construction, due to Milnor, of classifying spaces.  Using them, one can show that for an arbitrary group $G$, there exists a topological space $X$ with $\pi_1(X)\cong G$.

Various discrete analogues of fundamental groups for graphs appear in the literature.  This article deals with one such notion, namely that of the Matsushita discrete fundamental group.  Matsushita in fact defined \cite{matsushita} a family $\pi_1^r(X)$ of groups for a given graph $X$.  For $r=2$, in \cite{sankar}, Sankar gives an equivalent but purely combinatorial notion of discrete fundamental group.  (The same construction also appears in \cite{wrochna}.  We attribute it here to Matsushita, as that is the earliest appearance of it we have found to date.)  To take advantage of this purely combinatorial definition, we focus on $\pi_1^2(X)$ in this paper.  In \cref{sec:prelims}, we give all the precise definitions.

In \cite{sankar}, it is shown that $\pi_1^2(X)$ relates closely to the fundamental group of the neighborhood complex of $X$.  Moreover, $\pi_1^2$ is particularly well-suited for dealing with abelian Cayley graphs, as demonstrated in \cite{krebs-sankar}, wherein by using $\pi_1^2$ a complete characterization is found of those abelian groups that admit $3$-chromatic Cayley graphs.

This paper is structured as follows.  In \cref{sec:prelims}, we present the preliminary definitions and theorems needed for the subsequent sections.  In \cref{sec:svk}, we prove an analogue of the Seifert-van Kampen theorem.  (More precisely, we prove two such theorems, one for groupoids, and one for groups. These are \cref{thm:SVK-groupoids} and \cref{cor:SVK}, respectively.)  In \cref{sec:coverings}, we recap some results about graph coverings, mostly from \cite{matsushita}.  In \cref{sec:classifying}, we apply the results of \cref{sec:coverings} to prove \cref{thm:arbitrary-group}, which states that given an arbitrary group $G$, there exists a graph $X$ such that $\pi_1^2(X)\cong G$.

In \cref{sec:svk}, we assume the reader is familiar with basic concepts in category theory.  The other sections of this paper do not require any background in category theory.

Perhaps the most natural (and widely-studied) notion of fundamental group for a graph $X$ is to simply regard $X$ as a simplicial complex, with its edges as $1$-simplices, and then to define $\pi_1(X)$ to be the fundamental group of that simplicial complex.  It is well-known (see \cite{stillwell}, for example), that for a connected graph $X$, we have that $\pi_1(X)$ is a free group.  Our second main theorem contrasts sharply with this --- far from the severe restriction on what sort of group $\pi_1(X)$ can be (it must be free), $\pi_1^2(X)$ can be any group whatsoever.

Other versions of discrete fundamental groups for graphs appear in the literature, and other authors have proved results similar to ours for those.  One such version was studied by \cite{barcelo} and \cite{kapulkin}, among others.  In that version, to a given graph they associate a group $A_1(X)$.  In \cite{barcelo}, the authors obtain a Seifert van Kampen-type result for $A_1$, and in \cite{kapulkin}, this result is strengthened.  Using this stronger theorem, they prove that for any group $G$, there exists a graph $X$ such that $A_1(X)\cong G$. In \cref{sec:classifying}, we briefly discuss the relationship between $A_1$ and $\pi_1^2$, and we note that our construction yields an alternate proof of the theorem of \cite{kapulkin}.  (Essentially, our proof uses coverings, whereas their proof uses their analogue of the Seifert-van Kampen-theorem.)

We note that in \cite{grigoryan}, a Seifert van Kampen-type theorem is proved for yet another notion of discrete fundamental group.

Although to the authors' knowledge, there is not yet a generalization of $\pi_1^2$ to some form of higher homotopy groups for graphs, for a given graph $X$ there is however a family of groups $A_n(X)$.  In \cite{barcelo}, these are referred to as being akin to higher homotopy groups.  It might be interesting to know whether, given an abelian group $G$ and an integer $n\geq 2$, there exists a graph $X$ such that $A_n(X)\cong G$, and $A_j(X)$ is trivial for all $j\neq n$.  Perhaps such a construction might be modeled on that of Eilenberg-MacLane spaces, in the way that our construction in \cref{sec:coverings} is modeled on Milnor's construction of classifying spaces.

\section{Preliminaries}\label{sec:prelims}
    This section will start by introducing the notion of homotopy for graphs under consideration in this paper.  We follow here the presentation in \cite{sankar}.

\vspace{.05in}

    Given a graph $X$, a \textit{walk} of length $k \geq 0$ in $X$ is a sequence $P = v_0 ... v_k$ of $k+1$ vertices such that $v_i v_{i+1}$ is an edge of $X$ for each $0 \leq i < k$. $P$ is a \textit{closed walk} if its endpoints are equal, i.e., $v_0 = v_k$.  For a vertex $v$ in $X$, we let $N(v)$ denote the \emph{neighborhood} of $v$, that is, the set of all vertices in $X$ that are adjacent to $v$.
    \begin{definition}\label{def:homotopy}
        Let $X$ be a graph. Two walks $P$ and $P'$ of $X$ are \textit{homotopic} or \textit{homotopy equivalent} if the walk $P$ can be transformed into $P'$ via a finite sequence of steps: substitution, insertion, and deletion. These steps are defined as follows:
        \indent
        \begin{enumerate}
            \item Substitution: Given a walk $v_0 ... v_k$ and an index $0 < i < k$, replace $v_i$ with $v_i'$ for some $v_i' \in N(v_{i-1}) \cap N(v_{i+1})$.

            \item Insertion: Given a walk $v_0 ... v_k$ and an index $0 \leq i \leq k$, replace the single vertex $v_i$ with the three vertices $v_i w v_i$, for some $w \in N(v_i)$.

            \item Deletion: Given a walk $v_0 ... v_k$ and an index $0 < i < k$ where $v_{i-1} = v_{i+1}$, replace the three vertices $v_{i-1} v_i v_{i+1}$ with just $v_{i-1}$.
        \end{enumerate}
    \end{definition}

The similarities to topology should mostly be evident.  Walks in the graph will play the role of paths in topological spaces.  Homotopy equivalence of walks corresponds to path homotopy equivalence of paths.  (It follows from \cref{def:homotopy} that homotopic walks must have the same initial vertices as well as the same terminal vertices.  For that reason no distinction between homotopy and path homotopy needs to be made in the context of graphs.)  That insertion and deletion give rise to homotopies should be intuitive; substitution, less so.  In this regard it may help to think of the graph as a cellular complex, with vertices as $0$-cells, edges as $1$-cells, and $4$-cycles as $2$-cells.  See \cref{ex:cycle-def} below.  We contrast \cref{def:homotopy} to the notion of homotopy in \cite{barcelo} and \cite{kapulkin}, where in effect one regards both $3$-cycles and $4$-cycles as $2$-cells.

With this analogy in mind, the ``path'' (pun retroactively recognized but not originally intended) forward to define the discrete fundamental group is clear.  Elements will be closed walks at a given base vertex modulo homotopy equivalence, with concatenation as the group operation.  We now proceed to make this precise.

Let $v$ and $v'$ be vertices in a graph $X$.  It is straightforward to see that homotopy equivalence defines an equivalence relation on the set of walks with initial vertex $v$ and terminal vertex $v'$.  Given two walks $w_1=v_0\dots v_k$ and $w_2=v_k\dots v_{k+m}$ in $X$ where the initial vertex of $w_2$ equals the terminal vertex of $w_1$, recall that the \emph{concatenation of $w_1$ and $w_2$}, denoted $w_1*w_2$, is the walk $v_0\cdots v_{k+m}$ obtained by appending to the list of vertices in $w_1$ the list of all vertices (except the first) in $w_2$.  It is straightforward to see that if $w_1$ is homotopy equivalent to $w'_1$ and $w_2$ is homotopy equivalent to $w'_2$, then $w_1*w_2$ is homotopy equivalent to $w'_1*w'_2$.  Thus concatenation is well-defined on homotopy equivalence classes of walks.

\begin{definition}\label{def:fundamental-group}Let $X$ be a graph, and let $v_0$ be a vertex of $X$.  We define $\pi_1^2(X,v_0)$ to be the set of all homotopy equivalence classes of closed walks in $X$ based at $v_0$, equipped with the binary operation of concatenation.
\end{definition}

It is straightforward to prove that $\pi_1^2(X,v_0)$ is a group under concatenation.  Let $[w]$ denote the homotopy equivalence class of a walk $w$.  The identity element of $\pi_1^2(X,v_0)$ is $[v_0]$, i.e., the homotopy equivalence class of the trivial walk $v_0$.  The inverse of $[v_0v_1\dots v_{k-1}v_kv_0]$, namely $[v_0v_kv_{k-1}\dots v_1v_0]$, is formed by traversing the vertices in reverse order.

For a connected graph $X$, we define $\pi_1^2(X)$ to be $\pi_1^2(X,v_0)$ for an arbitrarily chosen vertex $v_0$ of $X$.  In \cite{sankar} it is shown that for any graph $X$ (whether connected or not), if there is a walk in $X$ from $v_0$ to $v'_0$, then $\pi_1^2(X,v_0)$ is isomorphic to $\pi_1^2(X,v'_0)$.  So for a connected graph $X$, it follows that $\pi_1^2(X)$ is well-defined up to group isomorphism.

We conclude this section with some examples.

\begin{example}\label{ex:path-def}
Let $P_2$ be a path graph of length $1$.  That is, $P_2$ has exactly two distinct vertices $a$ and $b$, and they are adjacent to each other.  A closed walk based at $a$ must be of the form $(ab)^na$ for some nonnegative integer $a$.  Performing $n$ deletions shows that this is homotopy equivalent to $a$.  Therefore $\pi_1^2(P_2)=1$.

More generally, let $P_n$ be a path graph of length $n-1$.  That is, $P_{n+1}$ has vertex set $\{0,1,\dots,n\}$, with edges $\{i,i+1\}$ for $i=0,\dots,n-1$.  Then $\pi_1^2(P_{n+1})=1$.  This can be shown directly from the definition, but an easier approach makes use of our version of the Seifert-van Kampen theorem.  We will return to this example in \cref{sec:svk}.\end{example}

\begin{example}\label{ex:cycle-def}
Let $C_n$ be an $n$-cycle, that is, the graph with $\mathbb{Z}/n\mathbb{Z}$ as its vertex set, with edges $\{i,i+1\}$ for $i=0,\dots,n-1$.

To compute $\pi_1^2(C_3,0)$, begin with an arbitrary closed walk $w$ in $C_3$ based at $0$, and then observe that by performing as many deletions as possible, $w$ is homotopy equivalent to $w_n$ for some integer $n$, where $w_n=(012)^n0$ if $n\geq 0$, and $w_n=(021)^{-n}0$ if $n\leq 0$.  With some work one can show that $[w_i]=[w_j]$ if and only if $i=j$.  Also, it is straightforward to compute that $[w_i]*[w_j]=[w_{i+j}]$.  It follows that $\pi_1^2(C_3)\cong\mathbb{Z}$.  This should be reminiscent of the computation of the fundamental group of the circle in topology.  In a similar way, one can show that $\pi_1^2(C_n)\cong\mathbb{Z}$ when $n\geq 5$.  In subsequent sections we show that---as with the circle in topology---these results can also be obtained via our Seifert-van Kampen-type theorems, as well as by using coverings.

Because we allow substitutions in the definition of homotopy, however, the computation for $C_4$ is different.  We begin similarly enough by showing that a closed walk $w$ in $C_4$ based at $0$ must be homotopy equivalent to $w_n$ for some integer $n$, where $w_n=(0123)^n0$ if $n\geq 0$, and $w_n=(0321)^{-n}0$ if $n\leq 0$.  However, we claim that for $n>0$, we have that $[w_n]=[0]$.  To see this, note that $0$ is a mutual neighbor of $1$ and $3$, so by performing substitutions, we may replace all instances of $2$ in $w_n$ with $0$s.  But then $w_n\sim (0103)^n0\sim(03)^n0\sim 0$.  Here $\sim$ denotes homotopy equivalence, and the last two equivalences involve repeated deletions.  We conclude that $\pi_1^2(C_4)=1$.

\end{example}

\section{Seifert-van Kampen}\label{sec:svk}

The Seifert-van Kampen theorem is a basic tool in algebraic topology for computing fundamental groups.  When a space $X$ equals a union of subspaces, under certain appropriate conditions this theorem expresses the fundamental group of a space in terms of the fundamental groups of those subspaces, taking into account the way in which they overlap.  The theorem is perhaps most naturally proven and stated using category-theoretic language, where the primary object is not the fundamental group but the fundamental groupoid.  The version of this theorem for groups falls out naturally from the more general version for groupoids, but the latter can deal with some computations that the former cannot handle --- this includes the ur-example of finding the fundamental group of a circle.

The statement and proof of our Seifert-van Kampen-type theorems for our fundamental groups and groupoids of graphs will be quite similar.  Mostly this will entail mimicry of the proofs for topological spaces.  The chief difference is in the condition imposed on subgraphs to ensure that the conclusion holds.  For topological spaces, one typically assumes that the various subspaces covering a given space $X$ are open subsets of $X$.  For graphs, we will instead assume that the various subgraphs covering a given graph $X$ contain among them every $4$-cycle in $X$.  Moreover, we will provide an example to show that the theorem can fail if this condition is not met.

Throughout this section we will assume that the reader is familiar with basic terminology and concepts from category theory as in \cite{bradley}, \cite{brown}.  This is the only section in this paper for which such a background is required.

\subsection{A Seifert-van Kampen-type theorem for Matsushita fundamental groupoids}

\begin{definition}
Let $X$ be a graph.  Let $A$ be a set of vertices of $X$.  We define the \emph{Matsushita fundamental groupoid of $X$ with base vertex set $A$}, denoted $\Pi(X,A)$, to be the category whose objects are the elements of $A$, and where the set of morphisms from an object $v_1$ to an object $v_2$ is the set of homotopy equivalence classes of walks in $X$ from $v_1$ to $v_2$.  Composition is defined by concatenation: $[w_2]\circ[w_1]:=[w_1*w_2]$.  (As discussed in \cref{sec:prelims}, this operation is well-defined.)

We define the \emph{Matsushita fundamental groupoid} of $X$, denoted $\Pi(X)$, by $\Pi(X,V)$, where $V$ is the vertex set of $X$.\end{definition}

Every morphism in $\Pi(X)$ is invertible (simply take the reverse walk), so $\Pi(X)$ is indeed a groupoid.  Moreover, it is straightforward to show that $\Pi$ defines a functor from the category of graphs (under graph homomorphisms) to the category $\mathsf{Grpd}$ of groupoids.

Observe that if $v_0$ is a vertex of $X$, then $\pi_1^2(X,v_0)$ is the group of morphisms of $\Pi(X,\{v_0\})$.

Our presentation for the remainder of this section closely follows the statement and proof of the classical Seifert-van Kampen theorem in \cite[p. 17]{mayconcise}.

Given categories $\mathcal{C}$ and $\mathcal{D}$, and functors $F_1,F_2:\mathcal{C}\to\mathcal{D}$, recall that a natural transformation $n$ from $F_1$ to $F_2$ associates with every object in $\mathcal{C}$ a morphism in $\mathcal{D}$ such that for any morphism $f:X\to Y$ in $\mathcal{C}$ the following diagram commutes

$$\begin{tikzcd}F_1(X)\arrow[r,"n_X"]\arrow[d,"F_1(f)"']&F_2(X)\arrow[d,"F_2(f)"]\\
F_1(Y)\arrow[r,"n_Y"]&F_2(Y)\end{tikzcd}$$

The category $\mathcal{D}^{\mathcal{C}}$ has as objects functors from $\mathcal{C}$ to $\mathcal{D}$, and natural transformations as morphisms. For any object $D\in\mathcal{D}$, the constant functor $\Delta D$ can be defined, which sends every object to $D$ and every morphism to $\text{id}_D$.

The colimit of a functor $F$, if it exists, is an object of $\mathcal{D}$ denoted $\varinjlim F$ and a natural transformation $c:F\to \Delta \varinjlim F$ such that for any other natural transformation $g:F\to\Delta G$ to a constant functor, there is a unique natural transformation $ \xi $ making the following diagram in $\mathcal{D}^{\mathcal{C}}$ commute

$$\begin{tikzcd}
F \arrow[d, "c"'] \arrow[dr,"g"] \\
 \Delta \varinjlim F\arrow[r, "\xi"']                  & \Delta G
\end{tikzcd}$$

The requirement is that given any morphism $f:X\to Y$ in $\mathcal{C}$, this diagram commutes:

$$
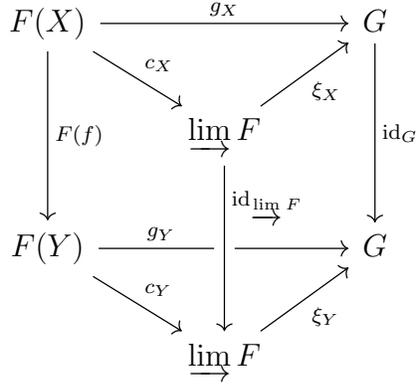
\begin{tikzcd}
F(X)
\arrow[dd,"F(f)"]
\arrow[rd, "c_X"]
\arrow[rr, "g_X"]&&
G\arrow[dd, "\text{id}_G"]\\&
\varinjlim F
\arrow[ru, "\xi_X"']\\
F(Y)
\arrow[rd, "c_Y"]
\arrow[rr, "g_Y" near start]&&
G\\&
\varinjlim F
\arrow[ru,"\xi_Y"']\arrow[from=uu, crossing over, "\text{id}_{\varinjlim F}" near start ]
\end{tikzcd}$$

In light of the identity morphisms above, $\xi_X=\xi_Y$. So, if every pair of objects in  $\mathcal{C}$ can be connected by a chain morphisms, then the data of $\xi$ consists of a single morphism (for which we will use the same symbol $\xi$) and the preceding diagram simplifies to

$$
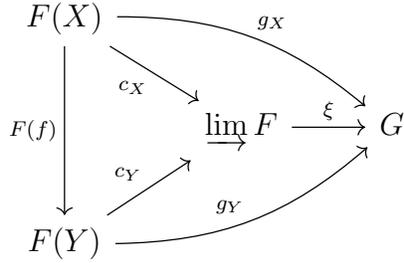
\begin{tikzcd}
F(X)\arrow[dd,"F(f)"'] \arrow[dr, "c_X"'] \arrow[drr, bend left=20, "g_X"] \\
 & \varinjlim F\arrow[r,"\xi"]& G\\
F(Y) \arrow[ur,"c_Y"]\arrow[urr, bend right=20, "g_Y"]
\end{tikzcd}$$

\begin{theorem}\label{thm:SVK-groupoids}Let $\mathcal{U}$ be a cover of a graph $X$ by subgraphs, closed under pairwise intersection, such that every $4$-cycle in $X$ is contained in some $U\in\mathcal{U}$. Let $\mathcal{U}$ also denote the subcategory of $\mathsf{Set}$ whose objects are elements of $\mathcal{U}$ and whose morphisms are inclusion maps. Let $\Pi$ denote the Matsushita fundamental groupoid functor restricted to $\mathcal{U}$:

$$\Pi:\mathcal{U}\to\mathsf{Grpd}$$ 

Then the colimit exists and is

$$\varinjlim \Pi=\Pi(X)$$

with the natural transformation from $\Pi$ to $\Delta \Pi(X)$ given by inclusion maps.

\end{theorem} 

\begin{proof} Note that any pair of objects $U,V\in\mathcal{U}$ can be connected by the morphisms $U\overset{\iota}{\longleftarrow} U\cap V\overset{\iota}{\longrightarrow} V$. So, we have to find a unique morphism $\xi:\Pi(X)\to G$ such that for all $U_1,U_2\in\mathcal{U}$ with $U_1\subset U_2$, the diagram (\ref{diagram}) in $\mathsf{Grpd}$ commutes.

\begin{equation}\label{diagram}\begin{tikzcd}
\Pi(U_1)\arrow[dd,"\iota"'] \arrow[dr, "\iota"'] \arrow[drr, bend left=20, "g_{U_1}"] \\
 & \Pi(X)\arrow[r,"\xi"]& G\\
\Pi(U_2) \arrow[ur,"\iota"]\arrow[urr, bend right=20, "g_{U_2}"]
\end{tikzcd}
\end{equation}

 Define $\xi$ as follows. For any object $v$ (ie. vertex in $X$), pick any $U$ that contains $v$ and define

$$\xi(v)= g _U(v)$$

The pairwise intersection property shows that this does not depend on which $U$ is used. For any identity morphism $\text{Id}_v$ (the trivial walk equivalence class at $v$), of course  define $\xi(\text{Id}_v)=\text{Id}_G$. For all remaining walk classes, pick some representative walk for $\gamma$

\begin{equation}\label{gamma}\gamma= [v_1\cdots v_n]
\end{equation}

and define

\begin{equation}\label{xiGamma}\xi(\gamma)=g_{U_1}([v_1v_{2}])\cdot g_{U_2}([v_2v_3])\cdots g_{U_{n-1}}([v_{n-1}v_n])
\end{equation}

where each $U_i\in\mathcal{U}$ contains $v_i$, $v_{i+1}$, and the edge between them. The pairwise intersection property shows that this also does not depend on the choices $U_i$, so it is well-defined for any representative walk. To show it is well-defined for each walk class, suppose $\gamma$ is modified by an insertion

\begin{equation}\label{gammaSecond}\gamma=[v_1\cdots v_iwv_i\cdots v_n]
\end{equation}

The vertices $v_i$ and $w$ and the edge between them are contained in some element of the cover $\tilde U$, so $\tilde U$ also contains the walk $v_iwv_i$ which is equivalent to $v_i$; so, noting that the maps $g_U$ are functors

$$g_{\tilde U}([v_iw])g_{\tilde U}([wv_i])= g _{\tilde U}([v_iwv_i])= g _{\tilde U}([v_i])=\text{Id}_G$$

So using either expression for $\gamma$ (equations \ref{gamma} and \ref{gammaSecond}) to define $\xi(\gamma)$ in the prescribed manner (equation \ref{xiGamma}) produces the same result. This also proves the same for deletions. The argument for substitutions is not that different; given two expressions for $\gamma$ obtained from each other by substitution

$$\gamma=[v_1\cdots v_iav_{i+1}\cdots v_n]=[v_1\cdots v_ibv_{i+1}\cdots v_n]$$

where $v_i,a,v_{i+1}$, and $b$ are the vertices of a $4$-cycle, now let $\tilde U$ be some element of the cover which contains these vertices, and so

$$ g _{\tilde U}([v_iav_{i+1}])= g _{\tilde U}([v_ibv_{i+1}])$$

Again using the functorial property of $g_{\tilde U}$ 

$$ g _{\tilde U}([v_ia])\cdot g _{\tilde U}([av_{i+1}])= g _{\tilde U}([v_ib])\cdot g _{\tilde U}([bv_{i+1}])$$

so again it doesn't matter which representative for $\gamma$ is used. Note that $\xi$ is indeed a functor; given two morphisms in $\Pi(X)$

$$\gamma_1=[v_1\cdots v_n]$$
$$\gamma_2=[w_1\cdots w_m]$$

$$\xi(\gamma_1\cdot \gamma_2)=\xi([v_1\cdots v_nw_1\cdots w_n])$$
$$=\xi([v_1v_2])\cdots\xi([v_{n-1}v_n])\xi([w_1w_2])\cdots\xi([w_{m-1}w_m])$$
$$=\xi([v_1\cdots v_n])\xi([w_1\cdots w_m])=\xi(\gamma_1)\xi(\gamma_2)$$

Finally, if $\gamma$ is some path class in some $U$, equation (\ref{xiGamma}) shows that the relevant diagram (\ref{diagram}) does commute for $\gamma$; on the other hand, that equation must in general hold if the diagrams (\ref{diagram}) commute, so $\xi$ is unique.\end{proof}

\subsection{A Seifert-van Kampen-type theorem for Matsushita fundamental groups}

We now use our Seifert-van Kampen theorem for Matsushita fundamental groupoids (\cref{thm:SVK-groupoids}) to prove a Seifert-van Kampen theorem for Matsushita fundamental groups (\cref{cor:SVK}).  More generally, we show that the theorem holds under appropriate conditions when the set of base vertices is a proper subset of the full vertex set; applying this to a single base vertex yields the theorem for groups.

\begin{theorem}\label{thm:SVK-set-of-base-vertices} Let $X$ be a graph, and let $\mathcal{U}$ be a cover of $X$ by subgraphs, closed under intersections, such that every $4$-cycle in $X$ is contained in some member of $\mathcal{U}$.  Let $A$ be a subset of the vertex set of $X$.  Suppose that every connected component of every member of $\mathcal{U}$ contains some element of $A$ as a vertex.  Let $\Pi(\cdot,A)$ be the restriction of the Matsushita fundamental groupoid functor to $\mathcal{U}$:

$$\Pi(\cdot,A):\mathcal{U}\to \mathsf{Grpd}$$

The colimit of $\Pi(\cdot,A)$ exists and 

$$\varinjlim \Pi(\cdot,A)=\Pi(X,A)$$

with natural transformation from $\Pi(\cdot,A)$ to $\Delta \Pi(X,A)$ given by inclusions.\end{theorem}

\begin{proof} Given a natural transformation $g$ from $\Pi(\cdot,A)$ to a constant functor $\Delta G$, we have to find a unique morphism $\xi:\Pi(X,A)\to G$ such that for all $U_1,U_2\in\mathcal{U}$ with $U_1\subset U_2$, the following diagram in $\mathsf{Grpd}$ commutes

\begin{equation}\label{diagram2}\begin{tikzcd}
\Pi(U_1,A)\arrow[dd,"\iota"'] \arrow[dr, "\iota"'] \arrow[drr, bend left=20, "g_{U_1}"] \\
 & \Pi(X,A)\arrow[r,"\xi"]& G\\
\Pi(U_2,A) \arrow[ur,"\iota"]\arrow[urr, bend right=20, "g_{U_2}"]
\end{tikzcd}
\end{equation}

We now construct a natural transformation $N$ from $\Pi$ to $\Pi(\cdot,A)$.  First, for each vertex $x\in X$, choose a walk class $\gamma_x$ as follows. If $x\in A$, simply choose $\gamma_x=\id_x$.  Otherwise select some vertex $a\in A$ such that $a$ is a vertex in the connected component of the intersection $U_0$ of all members of $\mathcal{U}$ containing $x$.  Thus there is a walk in $U_0$ from $a$ to $x$.  Choose $\gamma_x$ to be the homotopy equivalence class of one such walk.  Now for any morphism $\gamma$ in $\text{Hom}_{\Pi(U)}(c,d)$, we define $N_U(\gamma):=\gamma_d^{-1}\gamma\gamma_c$ in $\Pi(U,A)$.  It is straightforward to verify that $N$ is then a natural transformation, and that each $N_U$ restricts to the identity functor on $\Pi(U,A)$.

Consider the following diagram.

$$\begin{tikzcd}\Pi\arrow[r,"N"]\arrow[d,"\text{inclusions}"']&\Pi(\cdot,A)\arrow[d,"g"]\\
\Delta\Pi(X)\arrow[r,"\xi"]&\Delta G\end{tikzcd}$$

The composition $gN$ is then a natural transformation from $\Pi$ to $\Delta G$, so by \cref{thm:SVK-groupoids} there is a unique $\xi:\Delta\Pi(X)\to \Delta G$ making the diagrams (\ref{diagramGN}) commute. 

\begin{equation}\label{diagramGN}\begin{tikzcd}
\Pi(U_1)\arrow[dd,"\iota"'] \arrow[dr, "\iota"'] \arrow[drr, bend left=20, "g_{U_1}N_{U_1}"] \\
 & \Pi(X)\arrow[r,"\xi"]& G\\
\Pi(U_2) \arrow[ur,"\iota"]\arrow[urr, bend right=20, "g_{U_2}N_{U_2}"]
\end{tikzcd}
\end{equation}

Now consider the ``restrictions" of these diagrams to the subcategories $\Pi(U,A)$; since each $N_U$ is the identity on $\Pi(U,A)$, the result is the diagrams (\ref{diagram2}), which proves existence. For uniqueness, given any $\tilde \xi$ making the diagrams (\ref{diagram2}) commute, the diagrams could be ``extended" to diagrams (\ref{diagram}) by a similar procedure (i.e., using the $\gamma_x$ to extend each $g_U$ to the whole groupoid $\Pi(U)$); by the preceding theorem then the extensions of $\xi$ and $\tilde \xi$ must be the same, so $\xi=\tilde \xi$.\end{proof}

Taking $A$ to be a singleton set and covering $X$ by two subgraphs, we obtain the following corollary.

\begin{corollary}\label{cor:SVK}
Let $X$ be the union of two subgraphs $U_1$ and $U_2$.  Suppose that $X$, $U_1$, $U_2$, and $U_1\cap U_2$ are connected, and let $v_0$ be a vertex in $U_1\cap U_2$.  Then $\pi_1^2(X,v_0)$ is isomorphic to\[[\pi_1^2(U_1,v_0)*\pi_1^2(U_2,v_0)]/N,\]where $N$ is the normal subgroup of $\pi_1^2(U_1,v_0)*\pi_1^2(U_2,v_0)$ generated by all elements of the form \[(\iota_1)_*(\gamma)(\iota_2)_*(\gamma^{-1}),\]where $\iota_1$ and $\iota_2$ are inclusions of $U_1\cap U_2$ into $U_1$ and $U_2$, respectively.
\end{corollary}

\subsection{Examples}

\begin{example}
In \cref{ex:path-def}, we showed that $\pi_1^2(P_2)=1$.  Writing $P_{n+1}$ as a union of $P_2$ and $P_{n}$, it is straightforward to apply \cref{cor:SVK} inductively to conclude that $\pi_1^2(P_{n+1})=1$ for all positive integers $n$.
\end{example}

\begin{example}\label{ex:cycle-SVK}
We cannot use \cref{cor:SVK} to compute $\pi_1^2(C_n)$; the situation is much like that of the topological space $S^1$.  \cref{thm:SVK-set-of-base-vertices} will do the job, however, when $n\neq 4$.  Let $U_1$ be a path of length $1$ with ordered vertex set $[0,1]$.  Let $U_2$ be a path of length $n-1$ with ordered vertex set $[1,\dots,n-1,0]$.  Let $A=\{0,1\}$.  Then $\mathcal{U}=\{U_1,U_2,U_1\cap U_2\}$ satisfies the conditions of \cref{thm:SVK-set-of-base-vertices}, because $C_n$ contains no $4$-cycles.  It is straightforward to show that the colimit of $\Pi(\cdot,A)$ can be realized as a groupoid with two objects $0$ and $1$.  Moreover, $\pi_1^2(C_n,0)$ is the group of morphisms from $0$ to $0$ in $\varinjlim \Pi(\cdot,A)$, and this is infinite cyclic.  So $\pi_1^2(C_n)\cong\mathbb{Z}$ when $n\neq 4$.

In \cref{ex:cycle-def} we showed that $\pi_1^2(C_4)=1$.  Had we invalidly applied \cref{thm:SVK-set-of-base-vertices} here, we would have concluded incorrectly that $\pi_1^2(C_4)\cong\mathbb{Z}$.  This shows that we cannot omit the $4$-cycle condition from \cref{thm:SVK-set-of-base-vertices}.
\end{example}

\begin{example}
Let $X$ be the union of two cycles whose intersection is a single vertex, where neither cycle is a $4$-cycle.  It follows immediately from \cref{cor:SVK} and \cref{ex:cycle-SVK} that $\pi_1^2(X)\cong\mathbb{Z}*\mathbb{Z}$, the free product of $\mathbb{Z}$ with itself.  This is analogous to a wedge product of two circles.\end{example}

\section{Every group is $\pi_1^2(X)$ for some graph $X$}\label{sec:classifying}

The goal of this section is to prove the following theorem.

\begin{theorem}\label{thm:arbitrary-group}
    Let $G$ be an arbitrary group.  Then there exists a connected graph $X$ such that $\pi_1^2(X)\cong G$.
\end{theorem}

\subsection{Coverings}\label{sec:coverings}

In \cite{matsushita}, Matsushita develops a theory of coverings for the family $\pi_1^r(X)$ of discrete fundamental groups.  We begin by stating (without proof) a result from that article which we will require for our construction of a graph $X$ with prescribed $\pi_1^2(X)$.  We focus exclusively on $\pi_1^2(X)$, as those are the groups of interest in this paper.  The lemma we state here follows quickly from Theorem 6.11 in \cite{matsushita}.

\begin{lemma}\label{lem:group-action}Let $G$ be a group with identity element $e$.  Let $\tilde{X}$ be a connected graph such that $\pi^{2}_1(\tilde{X})=1$.  Suppose that $G$ acts freely on $\tilde{V}$ (the vertex set of $\tilde{X}$) by graph isomorphisms.  Also suppose that for all vertices $\tilde{v}\in\tilde{V}$ and all elements $g\in G\setminus{e}$, we have that there is no walk of length $1$, $2$ or $4$ from $\tilde{v}$ to $g\cdot \tilde{v}$.  Let $X=\tilde{X}/G$.  Then $\pi_1^{2}(X)\cong G$.\end{lemma}

Here $\tilde{X}/G$ denotes the quotient of $\tilde{X}$ by the action of $G$.  Its vertices are the orbits of the action of $G$.  Two orbits are adjacent if and only if one contains a vertex $\tilde{v}$ and the other contains some vertex $\tilde{w}$ adjacent to $\tilde{v}$.

\begin{example}
We return one last time to a computation of $\pi_1(C_n)$, this time using coverings.  Let $n\geq 3$ be a positive integer with $n\neq 4$.  Let $\tilde{X}$ be the ``doubly infinite path graph,'' that is, the graph with vertex set $\mathbb{Z}$ such that two vertices $x$ and $y$ are adjacent if and only if $|x-y|=1$.  We define an action of $\mathbb{Z}$ on $\tilde{X}$ by $k\cdot x=k+x$.  It is straightforward to verify that the conditions of \cref{lem:group-action} are met and that $\tilde{X}/\mathbb{Z}$ is the cycle graph $C_n$.  Therefore we once again find that $\pi_1^2(C_n)\cong\mathbb{Z}$.
\end{example}

\subsection{A construction of a graph $X$ with prescribed $\pi_1^2(X)$}

\begin{proof}[Proof of \cref{thm:arbitrary-group}]
Let $G$ be a group.  First assume that $|G|\geq 4$.  We will handle the special cases of $G\cong\mathbb{Z}/2\mathbb{Z}$ and $G\cong\mathbb{Z}/3\mathbb{Z}$ later.

We seek to use \cref{lem:group-action}.  Towards that end we will construct a graph $\tilde{X}$ satisfying the conditions of that lemma.  The graph $\tilde{X}$ will be constructed by stitching together various copies of the graph shown in \cref{fig:building-block}.  We now define the graph $B$.

Let $\text{Grid}(2,2)$ be the graph with vertex set $\{0,1,2\}\times\{0,1,2\}$, where two ordered pairs are adjacent if they are the same in one coordinate and differ by $1$ in the other coordinate.  That is, $\text{Grid}(2,2)$ is a $2\times 2$ grid graph.  Let $A_1, A_2$, and $A_3$ be disjoint graphs, each isomorphic to $\text{Grid}(2,2)$.  Denote the vertex $(x,y)$ of $A_j$ by $A_j(x,y)$.  We then define a graph $D$ to be the disjoint union of $A_1, A_2$, and $A_3$ modulo the following identifications.  We identify the vertices $A_1(0,2), A_1(1,2), A_1(2,2)$ (as well as the edges between them) of $A_1$ with the vertices $A_2(0,0), A_2(1,0), A_2(2,0)$, respectively, (as well as the edges between them) of $A_2$.  We identify the vertices $A_1(2,0), A_1(2,1), A_1(2,2)$ (as well as the edges between them) of $A_1$ with the vertices $A_3(0,0), A_3(0,1), A_3(0,2)$, respectively, (as well as the edges between them) of $A_3$.  We identify the vertices $A_2(2,0), A_2(2,1), A_2(2,2)$ (as well as the edges between them) of $A_2$ with the vertices $A_3(0,2), A_3(1,2), A_3(2,2)$, respectively, (as well as the edges between them) of $A_3$.  The resulting graph $D$ has $19$ vertices, which we label $D(A_j(x,y))$ for $j=1,2,3$ and $x,y=0,1,2$.

Let $D_1, D_2, D_3$ be disjoint graphs, each isomorphic to $D$.  The graph $B$ is defined to be the disjoint union of $D_1, D_2, D_3$, modulo the following identifications.  We identify the vertices $D_1(A_2(0,2))$, $D_1(A_2(1,2))$, $D_1(A_2(2,2))$, respectively (as well as the edges between them) of $D_1$ with the vertices $D_2(A_3(2,0))$, $D_2(A_3(2,1))$, $D_2(A_3(2,2))$, respectively (as well as the edges between them).  We identify the vertices $D_2(A_2(0,2))$, $D_2(A_2(1,2))$, $D_2(A_2(2,2))$, respectively (as well as the edges between them) of $D_1$ with the vertices $D_3(A_3(2,0))$, $D_3(A_3(2,1))$, $D_3(A_3(2,2))$, respectively (as well as the edges between them).  The graph $B$ is as shown in \cref{fig:building-block}.

\begin{center}
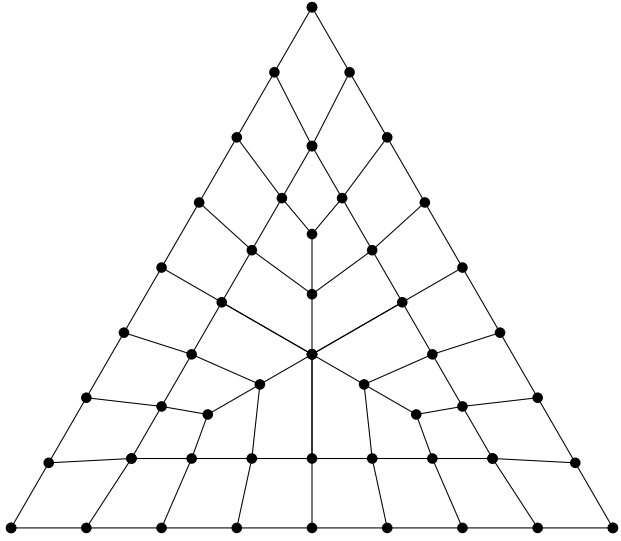
\begin{figure}
\begin{tikzpicture} 

\tikzmath{
\centX = 4;
\centY = 4 / sqrt(3);
\siz = .9;
\sizB = .8;
\sizC = .8;
}

\foreach \i in {0,...,8}{
    \coordinate (A\i) at (\i,0);
    \fill (A\i) circle (2pt);
}
\foreach \i in {0,...,8}{

    \coordinate (B\i) at ({\i * .5}, {\i * sqrt(3)/2});
    \fill (B\i) circle (2pt);
}
\foreach \i in {0,...,8}{
    \coordinate (C\i) at ({8 - \i * .5}, {\i * sqrt(3)/2});
    \fill(C\i) circle (2pt);
}

\foreach \i in {0,...,6}{
    \coordinate (D\i) at ({\i * \sizB + \centX - \sizB * 3}, {0 * \siz + \centY - 6/(2*sqrt(3)) * \sizB});
    \fill (D\i) circle (2pt);
}
\foreach \i in {0,...,6}{
    \coordinate (E\i) at ({\i * .5 * \sizB + \centX - \sizB * 3}, {\i * sqrt(3)/2 * \sizB + \centY - 6/(2*sqrt(3)) * \sizB});
    \fill (E\i) circle (2pt);    
}
\foreach \i in {0,...,6}{
    \coordinate (F\i) at ({-\i * .5 * \sizB + \centX + \sizB * 3}, {\i * sqrt(3)/2 * \sizB + \centY - 6/(2*sqrt(3)) * \sizB});
    \fill (F\i) circle (2pt);    
}

\foreach \i in {0,...,2}{
    \coordinate (G\i) at ({\centX}, {\i * \sizC + \centY});
    \fill (G\i) circle (2pt);
}
\foreach \i in {0,...,2}{
    \coordinate (H\i) at ({-\i * sqrt(3) / 2 * \sizC + \centX}, {-\i * .5 * \sizC + \centY});
    \fill (H\i) circle (2pt);}

\foreach \i in {0,...,2}{
    \coordinate (I\i) at ({\i * sqrt(3) / 2 * \sizC + \centX}, {-\i * .5 * \sizC + \centY});
    \fill (I\i) circle (2pt);
}

\foreach \i in {0,...,7}{
    \pgfmathtruncatemacro{\j}{\i+1}
    \draw (A\i) -- (A\j);
    \draw (B\i) -- (B\j);
    \draw (C\i) -- (C\j);
}
\foreach \i in {0,...,5}{
    \pgfmathtruncatemacro{\j}{\i+1}
    \draw (D\i) -- (D\j);
    \draw (E\i) -- (E\j);
    \draw (F\i) -- (F\j);
}
\foreach \i in {0,...,1}{
    \pgfmathtruncatemacro{\j}{\i+1}
    \draw (G\i) -- (G\j);
    \draw (H\i) -- (H\j);
    \draw (I\i) -- (I\j);
}

\foreach \i in {0,...,6}{
    \pgfmathtruncatemacro{\j}{\i+1}
    \draw (D\i) -- (A\j);
    \draw (E\i) -- (B\j);
    \draw (F\i) -- (C\j);
}

\foreach \i in {0,...,2}{
    \pgfmathtruncatemacro{\j}{\i+3}
    \draw (G\i) -- (E\j);
    \draw (G\i) -- (F\j);
}
\foreach \i in {0,...,2}{
    \pgfmathtruncatemacro{\j}{3 - \i}
    \draw (H\i) -- (D\j);
    \draw (H\i) -- (E\j);
}
\foreach \i in {0,...,2}{
    \pgfmathtruncatemacro{\j}{\i+3}
    \pgfmathtruncatemacro{\k}{3 - \i}
    \draw (I\i) -- (D\j);
    \draw (I\i) -- (F\k);
}

\end{tikzpicture}\caption{The basic ``building block'' graph $B$ for $\tilde{X}$}\label{fig:building-block}
\end{figure}
\end{center}

The key features of $B$ are that it is composed of $4$-cycles in such a way that $\pi_1^2(B)=1$; that it has three ``corners''; and that there is no walk of length $1$, $2$, or $4$ from a corner to any other corner.  Intuitively, we think of $B$ itself as being like a $2$-simplex (with the three corners as its vertices), and out of many copies of this ``$2$-simplex'' we build a graph that behaves something like a simplicial complex with an appropriate action of the group $G$.

For every triple $(g_1, g_2, g_3)$ of \emph{distinct} elements of $G$, let $B(g_1, g_2, g_3)$ be a copy of $B$ with $D_1$ denoted as $D_{12}(g_1, g_2, g_3)$ and $D_2$ denoted as $D_{23}(g_1, g_2, g_3)$ and $D_3$ denoted as $D_{31}(g_1, g_2, g_3)$.  Let $Y$ be the disjoint union of $B(g_1, g_2, g_3)$ over all such triples.  We form the graph $\tilde{X}$ by taking $Y$ modulo certain identifications that we now specify.

First, for all triples $g_1,g_2,g_3$ of distinct elements of $G$, we identify the entire graph $B(g_1, g_2, g_3)$ with $B(g_2, g_3, g_1)$ by identifying $D_{12}(g_1, g_2, g_3)$ with $D_{31}(g_2, g_3, g_1)$, and $D_{23}(g_1, g_2, g_3)$ with $D_{12}(g_2, g_3, g_1)$, and $D_{31}(g_1, g_2, g_3)$ with $D_{23}(g_2, g_3, g_1)$.

Next, suppose that $g_1, g_2, g_3, g_4$ are distinct elements of $G$.

Consider the walk $w(g_1,g_2)$ of length $8$ in $B(g_1, g_2, g_3)$ given by this sequence of vertices:

$D_{12}(g_1, g_2, g_3)(A_1(0,0))$

$D_{12}(g_1, g_2, g_3)(A_1(0,1))$

$D_{12}(g_1, g_2, g_3)(A_1(0,2))$

$D_{12}(g_1, g_2, g_3)(A_2(0,1))$

$D_{12}(g_1, g_2, g_3)(A_2(0,2))$

$D_{23}(g_1, g_2, g_3)(A_3(0,2))$

$D_{23}(g_1, g_2, g_3)(A_3(0,1))$

$D_{23}(g_1, g_2, g_3)(A_3(0,0))$

$D_{23}(g_1, g_2, g_3)(A_1(0,1))$

$D_{23}(g_1, g_2, g_3)(A_1(0,0))$

(Referring to \cref{fig:building-block}, we may think of this as the walk of length $8$ from the lower-left ``corner'' to the upper ``corner.'')  We identify that walk in $B(g_1,g_2,g_3)$, and the edges between those vertices, respectively, with the corresponding walk in $B(g_1,g_2,g_4)$, as well as the edges between those vertices, obtained by replacing all instances of $g_3$ with $g_4$.  If a vertex $v$ of $B(g_1,g_2,g_3)$ appears in such a walk $w$, we say that $v$ is in the \emph{fringe} of $B(g_1,g_2,g_3)$.  (Referring to \cref{fig:building-block}, the fringe of $B(g_1,g_2,g_3)$ is set of vertices appearing in the the outer triangle.  Because of the various identifications being made, the vertices in the fringe of are precisely those occurring in all three possible length-8 corner-to-corner walks in $B(g_1,g_2,g_3)$.)

Finally, suppose that $g_1, g_2, g_3, g_4, g_5$ are distinct elements of $G$.  We identify the vertex $D_{12}(g_1, g_2, g_3)(A_1(0,0))$ of $B(g_1, g_2, g_3)$ with the vertex $D_{12}(g_1, g_4, g_5)(A_1(0,0))$ of $B(g_1, g_4, g_5)$.

We define a group action of $G$ on the vertex set of $\tilde{X}$ by $g\cdot D_{jk}(g_1,g_2,g_3)(A_\ell(x,y))=D_{jk}(gg_1,gg_2,gg_3)(A_\ell(x,y))$.  One can show that this does in fact give us a well-defined group action.

We claim that with this group action, the graph $\tilde{X}$ so formed has the properties required to apply \cref{lem:group-action}.  We now prove this claim.  Towards this end, it is handy to introduce the notational convenience of denoting the vertex $D_{12}(g_1, g_2, g_3)(A_1(0,0))$ of $B(g_1, g_2, g_3)$ simply by $g_1$.  We refer to such a vertex as a ``corner'' of $B(g_1, g_2, g_3)$.  Because of the various identifications being made, we have that $B(g_1, g_2, g_3)$ has precisely three corners, namely, $g_1, g_2, g_3$.

First, we will show that $\tilde{X}$ is connected.  Because each graph $B(g_1,g_2,g_3)$ is connected, it suffices to show that there is a walk in $\tilde{X}$ from $a$ to $b$ for any two distinct elements $a,b\in G$.  By choosing $c\in G\setminus\{a,b\}$, it is clear that such a walk exists, because $B(a,b,c)$ is connected.

Next we show that $\pi^{2}_1(\tilde{X},e)=1$, where $e$ is the identity element of $G$.  Let $w$ be a closed walk in $\tilde{X}$ based at $e$.  We will show that $w$ is homotopy equivalent to the trivial walk.  First, by performing repeated substitutions (taking advantage of the various $4$-cycles in $B$), observe that $w$ is homotopy equivalent to a walk in which every vertex is in the fringe of some graph $B(g_1,g_2,g_3)$.  By performing deletions to eliminate all backtracking, we may then assume that $w$ is a concatenation of the walks $w(g_1,g_2)$ described earlier.  That is:\[w=w(g_0,g_1)*\cdots*w(g_{n-1},g_n)\]with $g_0=e=g_n$ and $g_i\neq g_{i+1}$ for all $i$.  Choose $n$ to be minimal amongst all such representations of $w$, up to homotopy equivalence.  We will show that we cannot have $n>0$, and thus that $w$ is homotopy equivalent to a trivial walk.  Suppose to the contrary that $n>0$.

If $g_i$, $g_{i+1}$, and $g_{i+2}$ are all distinct for some $i$, then by performing repeated substitutions in $B(g_i,g_{i+1},g_{i+2})$, we find that $w(g_i,g_{i+1})*w(g_{i+1},g_{i+2})$ is homotopy equivalent to $w(g_i,g_{i+2})$.  But this would reduce by $1$ the number of occurences of walks $w(a,b)$ in the expression for $w$, violating the minimality of $n$.  So this cannot occur.

We cannot have $n=1$, because $g_0\neq g_1$.  So $n\geq 2$, and $g_2=g_0$.  We claim that $w(g_0,g_1)*w(g_1,g_0)$ is homotopy equivalent to a trivial walk, once again violating minimality of $n$.  To see this, let $a,b$ be elements of $G$ distinct from $g_0, g_1$.  (Here is where we use that $|G|\geq 4$.)  From our previous argument, we see that $w(g_0,g_1)$ is homotopy equivalent to $w(g_0,a)*w(a,g_1)$ and that $w(g_1,g_0)$ is homotopy equivalent to $w(g_1,b)*w(b,g_0)$.  But then $w(g_0,g_1)*w(g_1,g_0)$ is homotopy equivalent to $w(g_0,a)*w(a,g_1)*w(g_1,b)*w(b,g_0)$, which in turn is homotopy equivalent to $w(g_0,a)*w(a,b)*w(b,g_0)$, which in turn is homotopy equivalent to a trivial walk, because $\pi_1^2(B(g_0,a,b))=1$.

This shows that $\pi^{2}_1(\tilde{X},e)=1$.

Finally, it is straightforward to show that $G$ acts freely on $\tilde{V}$ (the vertex set of $\tilde{X}$) by graph isomorphisms and that for all vertices $\tilde{v}\in\tilde{V}$ and all elements $g\in G\setminus\{e\}$, we have that there is no walk of length $1$, $2$ or $4$ from $\tilde{v}$ to $g\cdot \tilde{v}$.

Let $X$ be the quotient of $\tilde{X}$ by the action of $G$.  By \cref{lem:group-action}, we have that $\pi_1^2(X)\cong G$.

The special cases where $G=1$ or $G\cong\mathbb{Z}/2\mathbb{Z}$ or $G\cong\mathbb{Z}/3\mathbb{Z}$ can now be dealt with easily.  Simply embed $G$ as a subgroup of a larger group $H$ with $|H|\geq 4$, and construct $\tilde{X}$ as before.  Being a subgroup of $H$, we have that $G$ acts on $\tilde{X}$.  So we obtain the desired result by taking the quotient of $\tilde{X}$ by the action of $G$.\end{proof}

The articles \cite{barcelo} and \cite{kapulkin} deal with a different notion of discrete fundamental group of a graph $X$, denoted $A_1(X)$.  This is the fundamental group of the cellular complex formed by attaching a $2$-cell for each $3$-cycle and for each $4$-cycle.  In \cite{kapulkin} it is shown that for any group $G$, there exists a graph $X$ such that $A_1(X)\cong G$.  The graph $\tilde{X}$ in the proof of \cref{thm:arbitrary-group} contains no $3$-cycles.  Consequently, the proof of \cref{thm:arbitrary-group} applies equally well to $A_1$.  That is, for an arbitrary group $G$, there exists a graph $X$ such that $A_1(X)\cong G$, as previously shown in \cite{kapulkin}.

\bibliographystyle{amsplain}
\bibliography{SvK}

\end{document}